\documentclass[11pt]{amsart}
\usepackage{amsmath, amssymb, amsthm}
\usepackage[numbers]{natbib}
\usepackage{textcomp}
\usepackage{bibentry}
\usepackage{etoolbox}
\usepackage{hyperref}
\hypersetup{
  pdfauthor={Tianyuan Workshop},
  pdftitle={Open Problems in Computability Theory and Descriptive Set Theory},
  pdfsubject={Open Problems},
  colorlinks=true,
  linkcolor=blue,
  citecolor=blue,
  urlcolor=blue
}
\usepackage{lmodern}
\usepackage{tikz}
\usepackage{relsize}
\usepackage{parskip}
\usepackage{enumerate}
\usepackage{textcomp} 
\theoremstyle{definition}
\newtheorem{theorem}{Theorem}
\newtheorem{problem}[theorem]{Problem}
\newtheorem{proposition}[theorem]{Proposition}

\def\graph{\mathrm{GRAPH}}

\newcommand{\abstractbody}[1]{\noindent #1 \par \vspace{\baselineskip}}



\begin{document}

\title[Open Problems in Computability Theory and DST]{Open Problems in Computability Theory and Descriptive Set Theory}
\author[G. Barmpalias]{George Barmpalias}
\author[N. Bazhenov]{Nikolay Bazhenov}
\author[C. T. Chong]{Chi Tat Chong}
\author[W. Dai]{Wei Dai}
\author[S. Gao et al.]{Su Gao}
\author[]{Jun Le Goh}
\author[]{Jialiang He}
\author[]{Keng Meng Selwyn Ng}
\author[]{Andre Nies}
\author[]{Theodore Slaman}
\author[]{Riley Thornton}
\author[]{Wei Wang}
\author[]{Jing Yu}
\author[]{Liang Yu}

\date{June 2025}

\begin{abstract}
These open problems were presented in the Problem Sessions held during the Tianyuan Workshop on Computability Theory and Descriptive Set Theory, June 16-20, 2025. The problems are organized into sections named after their contributors, in the order of their presentations during the workshop. Notes were taken and compiled by Wei Dai, Feng Li, Ruiwen Li, Ming Xiao, Xu Wang, Víctor Hugo Yañez Salazar, and Yang Zheng.
\end{abstract}
\maketitle
\tableofcontents

\section{Theodore Slaman}
\abstractbody{
Given $A\subseteq \mathbb{R}$. Define
$$I(A)=\{f \text{ gauge function } : H^f(A)>0\}.$$

\begin{problem}
    Natural questions about $A \mapsto I(A)$:
    \begin{enumerate}
        \item Does the range of $I$ have the same cardinality as $2^{\mathbb{R}}$? (Seems yes under GCH.)
        \item If we consider the closed set $C$, $I(C)$ is arithmetic ($\mathbf{\Sigma}_2^0$). What's the complexity of the set
$$\big\{\mathbf{\Sigma}_2^0 \text{ formula } \varphi : \varphi \text{ defines an } I(C) \text{ for some } C\big\}?$$ Is it $\mathbf{\Sigma}_2^1$-complete? (It is $\mathbf{\Pi}_1^1$-hard.)
        \item Is the answer to "Borel hierarchy is proper w.r.t. the range of $I$" independent of ZFC?
        \item Is it consistent with ZFC that for every $f$ and $A\subseteq \mathbb{R}$, if $H^f(A)>0$, then there is a subset $A_0$ such that $H^f(A_0)$ is finite and positive?
    \end{enumerate}
\end{problem}
}
\begin{bibentry}{9}

\end{bibentry}

\section{George Barmpalias}
\abstractbody{
\begin{problem}
    Let $L_x=\{z : z \leq_K^+ x\}$ and $U_x=\{z : z \geq_K x\}$. Are the following true?
    \begin{enumerate}
        \item[(1)] $L_x$ is countable iff $\liminf_{n\to\infty}(K(x\upharpoonright n)-K(n))<\infty$;
        \item[(2)] $U_x$ is countable iff $\liminf_{n\to\infty}(K(n)+n-K(x\upharpoonright n))<\infty$.
    \end{enumerate}
\end{problem}

\begin{problem}
    Are the following true?
    \begin{enumerate}
        \item[(1)] For any c.e. set $A$, there exists a c.e. set $D\subseteq 2\mathbb{N}$, such that $A\equiv_{rK}D$ and $C(A\upharpoonright n|D\upharpoonright n)=C(D\upharpoonright n|A\upharpoonright n)=O(1)$;
        \item[(2)] For all $x$, there is a random $z$ such that $z \geq_{rK} x$ (or $z \geq_K x$).
    \end{enumerate}
\end{problem}

Consider the Even Number Game defined in \cite{Ba}. Given a natural number $k\geq 2$, the game $G_k$ is a game with two players who take turns to play natural numbers. Player 1 plays a set $A$ of $k$ integers where none of the members of $A$ has been played by her before; then Player 2 plays an even number between $\min(A)$ and $\max(A)$ (inclusive) which has not been played by him before. Player II loses when he has no legal moves.

It was shown in \cite{Ba} that for $k=2,3$, Player 1 has a winning strategy in the game $G_k$. For $k\geq 4$, this is unknown.

\begin{problem}
    For $k\geq 4$, does Player 1 have a winning strategy in the game $G_k$?
\end{problem}

\begin{bibentry}{9}

\end{bibentry}

}

\section{Jun Le Goh}
\abstractbody{
Recall that \textbf{WF} is the collection of all well-founded trees on $\omega$ and \textbf{UB} is the collection of all trees on $\omega$ which have a unique infinite branch.
It is known that \textbf{WF} and \textbf{UB} are both coanalytic complete.

Saint Raymond \cite{Raymond} proved in 2007 that (\textbf{WF}, \textbf{UB}) is a coanalytic complete pair using pure Descriptive Set Theory methods, i.e., for every disjoint pair $(A,B)$ of coanalytic sets in the Baire space $\omega^\omega$, there is a continuous function $f : \omega^\omega \to \omega^\omega$ such that $A = f^{-1}[\text{\textbf{WF}}]$ and $B = f^{-1}[\text{\textbf{UB}}]$.

Using recursion theory, Goh proved that any coanalytic separator of \textbf{WF} and \textbf{UB} is coanalytic complete. The method also proves Saint Raymond's theorem.
\begin{problem}
    Can this be proved using Descriptive Set Theoretic methods?
\end{problem}

\begin{bibentry}{9}

\end{bibentry}
}

\section{Liang Yu}{}
\abstractbody{
\begin{problem}[Downey--Hirschfeldt--Miller--Nies \cite{DHMN}]
    Is there an $x$ such that $\Omega^x \geq_T 0^{''}$ where
    $$\Omega^x = \sum_{\sigma \in 2^{<\omega}} 2^{-K^x(\sigma)}$$
    and
    $$K^x(\sigma) = \min\{|\tau| : U^x(\tau) = \sigma\}?$$
\end{problem}

Related to this question, Yu and Zhao proved the following proposition:

\begin{proposition}[Yu--Zhao]
    $\forall y \exists x \, \Omega^x \oplus 0^{'} \geq_T y$.
\end{proposition}

\begin{theorem}[Velickovic--Woodin \cite{VM}]
    If $A$ is $\mathbf{\Sigma}_1^1$ and $\sup\{\omega_1^{CK(x)} : x \in A\} = \omega_1$, then $\forall \gamma \exists x_0, x_1, x_2, x_3 \in A$,
    $$\gamma \leq_h x_0 \oplus x_1 \oplus x_2 \oplus x_3.$$
\end{theorem}
\begin{problem}
    \begin{enumerate}[(a)]
        \item Is there a recursion theoretical proof for the above theorem of Velickovic--Woodin?
        \item Is the following true: For any $\Sigma_1^1(y)$ set $A$, if there is $x \in A$ such that $\omega_1^{CK(x)} > \omega_1^{CK(y)}$, then the Velickovic--Woodin theorem holds for $A$?
    \end{enumerate}
\end{problem}

\begin{bibentry}{9}

\end{bibentry}

}

\section{Wei Wang}
\abstractbody{
We start with some definitions.

A tree $T \subseteq 2^{<\omega}$ is \emph{positive} if $[T]$ has measure $>0$.

Consider the following statements:
\begin{quote}
  $P^+$: For every positive tree $T$, there is a subtree $S \subseteq T$ which is perfect and positive.

  $P$: For every positive tree $T$, there is a subtree $S \subseteq T$ which is perfect.

  $P^-$: For every positive tree there are distinct $(X_n : n \in \omega)$ such that $X_n \in [T]$.
\end{quote}
It is known that
$$\text{WKL}_0 \vdash P^+$$
and
$$\text{RCA}_0 \vdash P^+ \to P \to P^- \to \text{1-RAN} \leftrightarrow \text{WWKL}_0,$$
where $\text{1-RAN}$ means there exists a 1-random. But we also know that
$$P^+ \nvdash \text{WKL}_0,$$
$$\text{RCA}_0 \nvdash P \to P^+,$$
and
$$\text{RCA}_0 \nvdash \text{1-RAN} \to P^-.$$

The questions are:
\begin{problem}
    Is it true that $\text{RCA}_0 \vdash P^- \to P$?
\end{problem}
\begin{problem}
    How about using dimension? (Define the above concepts by replacing measure by dimension.)
\end{problem}

Consider the following statements:
\begin{quote}
  $\text{PHP}(\Sigma_{n+1})$: There is no injective $F \in \Sigma_{n+1}$ such that for some positive $x$, $F : x+1 \to x$.

  $\text{WPHP}(\Sigma_{n+1})$: There is no injective $F \in \Sigma_{n+1}$ such that for some positive $x$, $F : 2x \to x$.

  $\text{GPHP}(\Sigma_{n+1})$: For any $x$, there is a $y$ such that there exists no $\Sigma_{n+1}$ injection with $F : y \to x$.

  $\text{GARD}(\Sigma_{n+1})$: For any $x$, there is no $\Sigma_{n+1}$ injection $F$ with $F : \mathbb{N} \to x$.
\end{quote}

It is known that, over $\text{I}\Sigma_n$,
$$\text{B}\Sigma_{n+1} \leftrightarrow \text{PHP}(\Sigma_{n+1}) \vdash \text{WPHP}(\Sigma_{n+1}) \to \text{GPHP}(\Sigma_{n+1}) \to \text{CARD}(\Sigma_{n+1}).$$
But we also know that
$$\text{PHP}(\Sigma_{n+1}) \nvdash \text{GPHP}(\Sigma_{n+1}) \to \text{WPHP}(\Sigma_{n+1})$$
and
$$\text{PHP}(\Sigma_{n+1}) \nvdash \text{CARD}(\Sigma_{n+1}) \to \text{GPHP}(\Sigma_{n+1}).$$

Finally, consider
\begin{quote}
  $\text{FRT}^e_k(\Sigma_{n+1})$: For any $x$, there is a $y$ such that every $\Sigma_{n+1} \, C : [y]^e \to k$ has a homogeneous set $H$ with $|H| \geq x$.
\end{quote}

It is known that
$$\text{I}\Sigma_n + \text{WPHP}(\Sigma_{n+1}) + \text{FRT}^e_k(\Sigma_{n+1})(e,k \text{ standard}) \nvdash \text{B}\Sigma_{n+1}$$
and
$$\text{FRT}^e_k(\Sigma_{n+1}) \nvdash \text{WPHP}(\Sigma_{n+1}).$$

The questions are:
\begin{problem}
    $\text{I}\Sigma_{n+1} + \text{WPHP}(\Sigma_{n+1}) \vdash \text{FRT}_2^2(\Sigma_{n+1})$? $\text{FRT}_2^2(\Sigma_{n+1}) \vdash \text{FRT}_2^3(\Sigma_{n+1})$?
\end{problem}
\begin{problem}
    The first-order theory \text{2-RAN} is between $\text{CARD}(\Sigma_2)$ and $\text{WPHP}(\Sigma_2)$, but how about $\text{GPHP}(\Sigma_2)$ and \text{2-RAN}?
\end{problem}
}

\section{Andre Nies}
\abstractbody{
\subsection*{Profinite groups}
A f.g. group $S = F_k/N$ is effectively residually finite (e.r.f.) if there is an algorithm that, on input $w \in F_k$, in case $w \notin N$ computes a finite group $Q$ and homomorphism $r : F_k \to Q$ such that $r(N) = \{e\}$ and $r(w) \neq e$.

\begin{theorem}
    A f.g. group $L$ is isomorphic to a subgroup of some computable profinite group that is generated by finitely many computable paths iff the following two conditions hold:
    \begin{enumerate}[(a)]
        \item $L$ has a $\Pi_1^0$ word problem (call this a $\Pi$-group)
        \item $L$ is effectively residually finite.
    \end{enumerate}
\end{theorem}

\begin{problem}
    \begin{enumerate}[(a)]
        \item Can such a group $L$ have unsolvable word problem?
        \item Is there a f.g., residually finite $\Pi$-group that is not effectively r.f.?
        \item Morozov (Higman’s question revisited, 2000) constructed a $\Pi$-group that is not isomorphic to a subgroup of $S_{\text{rec}}$. Can we make such a group r.f.? It can't be e.r.f.
    \end{enumerate}
\end{problem}

For each left $\Sigma_2$ real $r \in [0, 1]$ there is a computable profinite group and computable subgroup with Hausdorff dimension $r$.

\begin{problem}
    Which values can occur for $r$ when $G$ is also topologically finitely generated?
\end{problem}

Nies, Segal, and Tent 2021 \cite{NST} studied expressiveness of f.o. logic for profinite group. Many such groups are axiomatised by single sentence in the language of groups, among this reference class.

\begin{problem}
    For a f.o. sentence $\varphi$, what is the possible complexity of $\{G : G \models \varphi\}$?
\end{problem}

The following problem was also posed during the workshop, and by the end of the workshop Gao and Nies realized that the answer is yes.

\begin{problem}
    Is there some Borel class of profinite groups with isomorphism relation properly between $S_\infty$-complete and smooth? In particular, how about the Abelian case?
\end{problem}
Update: The answer to this problem is yes. Due to the Pontryagin duality, the isomorphism relation of profinite abelian groups is Borel bireducible with the isomorphism of countable torsion abelian groups. It is well known that the latter relation is classified by the Ulm invariants, and is known to be $\mathbf{\Sigma}_1^1$-complete. In terms of Borel reducibility, it is strictly above the smooth equivalence relation and strictly below the graph isomorphism; in particular, it is not above $E_0$.

One could still consider other classes of profinite groups, such as small (only finitely many subgroups of each finite index), or strongly complete (each subgroup of finite index is open). See Dan Segal's work for references.

\subsection*{Complexity of isomorphism of oligomorphic groups}
\begin{problem}
    Nies, Schlicht, and Tent proved in 2019 \cite{NST2} that the topological isomorphism of oligomorphic groups is $\leq_B E_\infty$. Is it smooth?
\end{problem}

\begin{theorem}[Nies--Paolini \cite{NP}]
    \begin{enumerate}[(1)]
        \item Let $G$ be a non-archimedean Roelcke precompact group. Then the group $\text{Aut}(G)$ of continuous automorphisms of $G$ carries a natural Polish topology and $\text{Inn}(G)$ is closed in it.
        \item Suppose in addition that $G$ is oligomorphic, then $\text{Out}(G) = \text{Aut}(G)/\text{Inn}(G)$ with the quotient topology is totally disconnected, locally compact (t.d.l.c.).
    \end{enumerate}
\end{theorem}
\begin{problem}
    If $G$ is oligomorphic, is $\text{Out}(G)$ always profinite?
\end{problem}

\begin{problem}
    Is there a Borel, invariant class $C$ of closed subgroups of $S_\infty$ with isomorphism relation not Borel below graph isomorphism? Analytic complete? How about the class $C$ of pro-countable groups?
\end{problem}

\subsection*{Computability theory}
In \cite[Section 3.1]{LMNS} the authors showed the following implications
$$\text{1-generic } \Delta^0_2 \Rightarrow \text{ index guessable } \Rightarrow \text{ computes no maximal tower } \Rightarrow \text{ low }$$

\begin{problem}
    In the above diagram, are the first and the second implications reversible?
\end{problem}

\begin{bibentry}{9}

\end{bibentry}
}

\section{Nikolay Bazhenov}
\abstractbody{
Let $x$ be a Turing degree. A computable structure $S$ is $x$-computably categorical if for any computable structure $A \cong S$, there exists an $x$-computable isomorphism $f : A \to S$.

A Turing degree $d$ is the degree of categoricity for $S$ if $d$ is the least degree such that $S$ is $d$-computably categorical.

A computable structure $S$ is decidable if its complete diagram is computable: i.e., given a first-order formula $\psi(\bar{x})$ and a tuple $\bar{a}$ from $S$, one can computably check whether $S \models \psi(\bar{a})$.

A decidable structure $S$ is decidably $x$-categorical if for any decidable structure $A \cong S$, there exists an $x$-computable isomorphism $f : A \to S$.

A degree $d$ is the degree of decidable categoricity for $S$ if $d$ is the least degree such that $S$ is decidably $d$-categorical.

\begin{problem}
    Every degree of decidable categoricity is a degree of categoricity. Is the converse true?

    A conjecture for this problem is that it is true for $d \geq 0^{(\alpha)}$, where $\alpha$ is 'sufficiently large' computable ordinal. Say, for $\alpha = \omega$. What happens with d.c.e. degrees $d$?
\end{problem}

Ceer stands for computably enumerable equivalence relation on $\omega$. A diagonal function (or a fixed-point-free function) for a ceer $E$ is a total function $g(x)$ satisfying $\neg g(x) E x$ for all $x \in \omega$.

\begin{problem}
    The following are three versions of this problem.
    \begin{enumerate}[(i)]
        \item Describe the class \textbf{Diag} containing those Turing degrees $d$ such that every ceer $E \neq \text{Id}_1$ admits a $d$-computable diagonal function.
        Here we have some partial results:
        \begin{enumerate}[(a)]
            \item $\text{\textbf{PA}} \subseteq \text{\textbf{Diag}} \subseteq \text{\textbf{DNR}}$ [Badaev, Bazhenov, Kalmurzayev, Mustafa 2024].
            \item There exists a Martin-Löf random $x$ such that $x \notin \text{\textbf{Diag}}$ [Ng].
        \end{enumerate}
        \item What is the reverse-mathematical strength of the statement “Every $\Sigma_1^0$-definable equivalence relation $E \neq \text{Id}_1$ has a diagonal function”?
        \item What is the Weihrauch degree of the corresponding problem?
    \end{enumerate}
\end{problem}
}

\section{Keng Meng Selwyn Ng}
\abstractbody{
A right-c.e. metric space $(S, d)$ consists of a countable set $S = \{c_0, c_1, \dots\}$ and $d : \mathbb{N}^2 \to \mathbb{R}$ such that $d(c_i, c_j)$ is a right-c.e. real number, uniformly in $i, j$. Same for a left-c.e. metric space.
\begin{problem}
    Does every (effectively) compact left-c.e. Polish space have a computable (right c.e.) Polish copy? Related to this problem, Melnikov and Ng proved that there is a left c.e. Polish space with no computable Polish presentation. Similarly, Koh, Melnikov, and Ng proved that there is a left c.e. Polish space with no right c.e. Polish presentation.
    As a positive result, Melnikov and Ng proved that every left c.e. Stone space is homeomorphic to a computable Polish space.
\end{problem}

A computable Polish space $X$ is $\alpha$-categorical if for every pair of computable metric spaces $M$, $N$ such that $\bar{M} \cong \bar{N} \cong X$, there is
an $\alpha$-computable homeomorphism between $\bar{M}$ and $\bar{N}$. The least Turing degree $\alpha$ is the degree of (topological) categoricity of $X$.
$X$ is (topologically) computably categorical if it is $0$-categorical.

\begin{problem}
    Is there an infinite computable Polish space that is topologically computably categorical?
\end{problem}
\begin{problem}
    Does the Baire space have a degree of categoricity?
\end{problem}
\begin{problem}
    Is every/any c.e. degree the degree of topological categoricity?
\end{problem}
}

\section{Chong Chi Tat}
\abstractbody{
Let $M$ be a model satisfying $\text{RCA}_0 + \text{B}\Sigma_n + \neg \text{I}\Sigma_n$. Fix a $\Pi_n$ subset $X$ of $M$ and assume $\langle A_s : s \in M \rangle$ is a sequence of pairwise disjoint $M$-finite sets such that $\forall s \, (A_s \cap X \neq \emptyset)$.

\begin{problem}
    Is there a definable in $M$ choice function $f$ of $\langle A_s \rangle$, i.e., $f(s) \in A_s \cap X$?
\end{problem}

Tree version: Without loss of generality, let us state the problem for $n=2$.

\begin{problem}
    Consider a tree $T$ which is $\Pi^0_1(\emptyset')$ in $M$. Assume there is a $g \in [T]$ such that $M[g] \models \text{B}\Sigma_2$. Is there a definable path $h \in [T]$ such that $M[h] \models \text{B}\Sigma_2$?
\end{problem}
}

\section{Su Gao}
\abstractbody{
\subsection*{The graph complement problem}
In 2021, Matt Foreman asked the following question:
\begin{problem}
    Is the set $\{G \in 2^{\omega \times \omega} : G \text{ is a graph on } \omega \text{ and } G \cong G^c\}$ Borel?
\end{problem}
Here $G^c = \{(m,n) \in \omega \times \omega : m \neq n \land (m,n) \notin G\}$ is the complement graph of $G$.

The problem was solved during the workshop.

\begin{theorem}[Feng Li--Ruiwen Li--Ming Xiao; Riley Thornton]
    The set
    $$\{G \in 2^{\omega \times \omega} : G \text{ is a graph on } \omega \text{ and } G \cong G^c\}$$
    is $\boldsymbol{\Sigma}_1^1$-complete.
\end{theorem}

\begin{proof}
    We define a continuous function $\Phi$ from $\graph^2$ to $\graph$, such that
    \begin{align*}
        (m,n) \in \Phi(G,H) \iff & m = 4k_1, n = 4k_2 \text{ and } (k_1,k_2) \in G; \\
        & m = 4k_1+1, n = 4k_2+1 \text{ and } (k_1,k_2) \in H^c; \\
        & m = 4k_1+2, n = 4k_2+2 \text{ and } (k_1,k_2) \in H^c; \\
        & m = 4k_1+3, n = 4k_2+3 \text{ and } (k_1,k_2) \in G; \\
        & m = 4k_1, n = 4k_2+1; \\
        & m = 4k_1+1, n = 4k_2+2; \\
        & m = 4k_1+2, n = 4k_2+3;
    \end{align*}

    The following figure depicts our definition for $\Phi$:
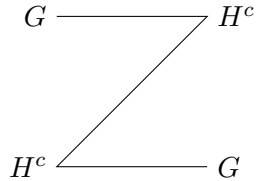
    \begin{figure}[htbp]
        \centering
        \begin{tikzpicture}
            \draw (0,0) node[left] {$G$} -- (2,0) node[right] {$H^c$} -- (0,-2) node[left] {$H^c$} -- (2,-2) node[right] {$G$};
        \end{tikzpicture}
    \end{figure}

    Now if $G \cong H$, it is easy to see $\Phi(G,H)$ is isomorphic to its complement.

    Conversely, if $\Phi(G,H) \cong \Phi(G,H)^c$, let $f$ be the witness of their isomorphism, and $d_1, d_2$ be the distance function on $\Phi(G,H)$ and $\Phi(G,H)^c$ respectively. Note that in $\Phi(G,H)$, $4\mathbb{N} \cup 4\mathbb{N}+3 = \{x \in \omega : \exists y \, d_1(x,y) = 3\}$. This is similar for $\Phi(G,H)^c$. So $f$ sends $4\mathbb{N} \cup 4\mathbb{N}+3$ to $4\mathbb{N}+1 \cup 4\mathbb{N}+2$. Also note that
    $$\{x \in \omega : \exists y \, (d_1(x,y) = 3) \land d_1(x,0) \leq 2\}$$
    is a copy of $G$,
    $$\{x \in \omega : \exists y \, (d_2(x,y) = 3) \land d_2(x,f(0)) \leq 2\}$$
    is a copy of $H$, and $f$ maps the former set to the latter set. So $G \cong H$.
\end{proof}

Using similar methods, Feng Li and Ruiwen Li also obtained the following:

\begin{theorem}
    The following sets are $\boldsymbol{\Sigma}^1_1$-complete:
    \begin{enumerate}
        \item The set of all countable directed graphs $G$ where $G \cong G^c$;
        \item The set of all countable tournaments $G$ where $G \cong G^c$.
    \end{enumerate}
\end{theorem}

A directed graph $G$ is a \emph{tournament} if for any distinct vertices $x, y \in G$, exactly one of $(x,y)$ and $(y,x)$ is an edge in $G$.

\subsection*{Some hyperfiniteness problems}
The following is the well-known Union Problem in the theory of hyperfinite equivalence relations.

\begin{problem}[The union problem]
    If $E$ is a countable Borel equivalence relation and $E = \cup_n F_n$, where $F_n$ is hyperfinite and $F_n \subseteq F_{n+1}$ for each $n \in \omega$, is it true that $E$ is hyperfinite?
\end{problem}

The following theorem was recently proved.

\begin{theorem}[Frisch--Shinko--Vidnyánszky \cite{FSV}]
    If there is a counterexample to the union problem, then the set of all hyperfinite equivalence relations is $\mathbf{\Sigma}_2^1$-complete.
\end{theorem}

Analogous to the above, we call a countable Borel equivalence relation \emph{hyperfinite-over-hyperfinite} (or shortly, \emph{hf/hf}) if there is a Borel partial order $\leq$ on $X$ such that
\begin{enumerate}[(i)]
    \item if $x \leq y$, then $x E y$,
    \item the order type of $\leq \upharpoonright_{[x]}$ is a suborder of $\mathbb{Z}^2$.
\end{enumerate}

\begin{problem}[The hf/hf problem]
    If $E$ is hf/hf, is it hyperfinite?
\end{problem}

There is an equivalent characterization for hyperfiniteness of hf/hf equivalence relation:
\begin{theorem}[Gao--Xiao \cite{GX}]
    If $E$ is hf/hf, then $E$ is hyperfinite iff $E$ admits a $\mathbb{Z}^2$ ordering which is self-compatible.
\end{theorem}

Inspired by the Frisch--Shinko--Vidnyánszky theorem, we ask:

\begin{problem}
    Suppose there is a counterexample to the hf/hf problem. Is the set of all hyperfinite equivalence relations $\mathbf{\Sigma}_2^1$-complete?
\end{problem}

\begin{bibentry}{9}

\end{bibentry}
}

\section{Jing Yu}
\abstractbody{
\begin{problem}[Weiss' question]
    Is the orbit equivalence relation of Borel action of countable amenable group hyperfinite?
\end{problem}

The orbit equivalence relations of actions of the following groups are proved to be hyperfinite:
\begin{quote}
    $\mathbb{Z}$ (Slaman--Steel),

    $\mathbb{Z}^n$ (Weiss),

    Groups of polynomial growth (Jackson--Kechris--Louveau),

    Abelian groups (Gao--Jackson),

    Locally nilpotent groups (Seward--Schneider),

    Polycyclic groups (Conley--Jackson--Marks--Seward--Tucker-Drob).
\end{quote}

Also the connected equivalence relations of the following graphs are proved to be hyperfinite:
\begin{quote}
    Graphs of polynomial growth (Bernshteyn--Yu),

    Graph of growth less than $\exp(r^c)$ for some small enough $0<c<1$ (Grebík--Marks--Rozhoň--Shinko).
\end{quote}

\begin{problem}
    How about graphs of uniform subexponential growth, i.e., of growth less than $\exp(r^c)$ for some $0<c<1$?
\end{problem}
}

\section{Riley Thornton}
\abstractbody{
Let $\text{Aut}([0,1],\lambda)$ denote the space of all automorphisms of the Lebesgue measure with the weak topology. Consider actions of $F_\infty$ on $[0,1]$ by automorphisms of Lebesgue measure.
Then the space of all such actions, $\text{Act}(F_\infty,([0,1],\lambda))$, is a closed subspace of $\text{Aut}([0,1],\lambda)^{F_\infty}$. 

For $x\in \text{Act}(F_\infty, ([0,1],\lambda))$, let $\mbox{Sch}(x)$ denote the Schreier graph of the action on $[0,1]$.

\begin{problem}
   Is the set
   $$\{x \in \text{Act}(F_\infty,([0,1],\lambda)) : \text{Sch}(x) \text{ has a measurable perfect matching}\}$$ 
   $\mathbf{\Sigma}_1^1$-complete?
\end{problem}

\begin{problem}[Halmos, 1956]
    What $T \in \text{Aut}([0,1],\lambda)$ has $S \in \text{Aut}([0,1],\lambda)$ with $T = S^2 = S \circ S$?

    Conjecture: $\{T \in \text{Aut}([0,1],\lambda) : \exists S \in \text{Aut}([0,1],\lambda) \, T = S^2\}$ is $\mathbf{\Sigma}_1^1$-complete.
\end{problem}
}

\section{Wei Dai}
\abstractbody{
\begin{problem}
    If $\vec{G}$ is a Borel locally finite directed graph whose in-degrees are uniformly bounded and out-neighborhoods have polynomial growth, is the connected equivalence relation of $\vec{G}$ hyperfinite?
\end{problem}

\begin{problem}
    Let $\alpha$ be a free, probability measure preserving (p.m.p.) and ergodic action of $\mathbb{F}_2$ on measure space $(X,\mu)$.
    \begin{enumerate}[(i)]
        \item If $e \neq x \in \mathbb{F}_2$, is there a free, p.m.p action $\beta$ such that $E_\alpha = E_\beta$ and $x$ acts ergodically?
        \item (Miller--Tserunyan). Is there a free, p.m.p action $\beta$ such that every nontrivial element of $\mathbb{F}_2$ acts ergodically?
    \end{enumerate}
\end{problem}
}

\section{Jialiang He}
\abstractbody{
A maximal eventually different family (MED) is a family $\mathlarger{\mathlarger{\varepsilon}} \subseteq \omega^\omega$ such that:
\begin{enumerate}[(i)]
    \item $\forall f \neq g \in \mathlarger{\mathlarger{\varepsilon}} \, |f \cap g| < \infty$ and
    \item $\forall f \in \omega^\omega \exists g \in \mathlarger{\mathlarger{\varepsilon}} \, |f \cap g| = \infty$.
\end{enumerate}

\begin{theorem}
    There exists a closed MED.
\end{theorem}

(Gao asked if there is a $\Pi^0_1$ MED.)

Let $I$ be an ideal on $\omega$. We can define $I$-MED analogous to the above definition. A family $\mathlarger{\mathlarger{\varepsilon}} \subseteq \omega^\omega$ is called an $I$-MED if:
\begin{enumerate}[(i)]
    \item For all $f \neq g \in \mathlarger{\mathlarger{\varepsilon}}$, $\{n : f(n) = g(n)\} \in I$ and
    \item $\forall h \in \omega^\omega \exists f \in \mathlarger{\mathlarger{\varepsilon}} \, \{n : f(n) = g(n)\} \in I^+$.
\end{enumerate}

\begin{theorem}
    If $I$ is an $F_\sigma$ ideal or $I = \text{Fin}^\alpha$, $\alpha < \omega_1$, then there exists a closed $I$-MED.
\end{theorem}

\begin{problem}
    For all Borel $I$, is there a closed $I$-MED?
\end{problem}

\begin{problem}
    For maximal ideal $I$, is there a closed $I$-MED?
\end{problem}

For these problems, it is enough to find a Borel $I$-MED.
\begin{theorem}
    For any ideal $I$, if there is a Borel $I$-MED, then there is a closed $I$-MED.
\end{theorem}
}
\end{document}